\documentclass[letterpaper,10 pt,conference]{ieeeconf}

\usepackage{amsmath,amssymb,amsfonts}
\usepackage{cite}
\usepackage{graphicx}
\usepackage{mathtools}
\usepackage{textcomp}
\usepackage{tikz}
\usepackage{xcolor}

\usetikzlibrary{arrows,automata}

\def\BibTeX{{\rm B\kern-.05em{\sc i\kern-.025em b}\kern-.08em T\kern-.1667em\lower.7ex\hbox{E}\kern-.125emX}}
\DeclareMathOperator{\rank}{rank}

\newcommand{\bbR}{\mathbb{R}}
\newcommand{\bbC}{\mathbb{C}}

\newcommand{\calA}{\mathcal{A}}
\newcommand{\calB}{\mathcal{B}}
\newcommand{\calC}{\mathcal{C}}
\newcommand{\calD}{\mathcal{D}}
\newcommand{\calE}{\mathcal{E}}

\newcommand{\calP}{\mathcal{P}}

\newcommand{\calI}{\mathcal{I}}

\newcommand{\set}[2]{\left\{ #1 ~\left|~ \vphantom{#1} #2 \right. \right\}}
\newcommand{\nset}[1]{\{1,\dots,#1\}}
\newcommand{\qand}{\quad\text{and}\quad}

\newcommand{\bbm}{\begin{bmatrix}}
\newcommand{\ebm}{\end{bmatrix}}
\newcommand{\tstrut}{\rule{0pt}{2.2ex}}

\newcommand{\pattern}{\{0,*,{?}\} }
\newcommand{\sys}{(A,B,C,D)}
\newcommand{\patsys}{{(\calA, \calB, \calC, \calD)}}
\newcommand{\classprod}{$A\in\calP(\calA)$, $B\in \calP(\calB)$, $C\in \calP(\calC)$ and {$D\in \calP(\calD)$}}
\newcommand{\daesys}{(E,A,B)}
\newcommand{\daepatsys}{(\calE,\calA,\calB)}
\newcommand{\daeclassprod}{$E\in\calP(\calE)$, $A\in \calP(\calA)$ and $B\in\calP(\calB)$}

\newcommand{\daeinputs}[1]{C_\textup{p}^{#1}}

\renewcommand{\t}{^\top}
\newcommand{\s}{^\ast}
\newcommand{\inv}{^{-1}}
\IEEEoverridecommandlockouts

\newtheorem{lemma}{Lemma}
\newtheorem{proposition}{Proposition}
\newtheorem{remark}{Remark}
\newtheorem{definition}{Definition}
\newtheorem{example}{Example}
\newtheorem{theorem}{Theorem}

\title{\LARGE{Properties of pattern matrices with applications to structured systems}}
\author{B. M. Shali, H. J. van Waarde, M. K. Camlibel, \emph{Member, IEEE}, and H. L. Trentelman, \emph{Fellow, IEEE} %
\thanks{The authors are with the Bernoulli Institute for Mathematics, Computer Science, and Artificial Intelligence, University of Groningen, Nijenborgh 9, 9747 AG, Groningen, The Netherlands. H. J. van Waarde is  also  with  the  Engineering  and  Technology  Institute  Groningen, University  of  Groningen,  Nijenborgh  4,  9747  AG,  Groningen, The Netherlands. (email: \texttt{b.m.shali@rug.nl}, \texttt{h.j.van.waarde@rug.nl}; \texttt{m.k.camlibel@rug.nl}; \texttt{h.l.trentelman@rug.nl}).}}

\begin{document}
\maketitle

\begin{abstract}
	The exact parameter values of mathematical models are often uncertain or even unknown. Nevertheless, we may have access to crude information about the parameters, e.g., that some of them are nonzero. Such information can be captured by so-called pattern matrices, whose symbolic entries are used to represent the available information about the corresponding parameters. In this paper, we focus on pattern matrices with three types of symbolic entries: those that represent zero, nonzero, and arbitrary parameters. We formally define and study addition and multiplication of such pattern matrices. The results are then used in the study of three strong structural properties, namely, controllability of linear descriptor systems, and input-state observability and output controllability of linear systems.
\end{abstract}


\section{Introduction}\label{sec:introduction}
The concept of structure was introduced more than 40 years ago by Lin \cite{lin1974} in order to obtain more realistic models of physical systems. Lin considers a linear time-invariant system with a single input, where the numerical entries of the system matrices are not known precisely, but are known to be fixed zeros or arbitrary real numbers. This pattern of fixed zero entries gives the system its \emph{structure}, and is what makes it a \emph{structured system}. Based on this concept of structure, a structured system is said to be \emph{(weakly) structurally controllable} if there exists at least one controllable system with the given structure, i.e., a controllable system where the numerical entries of the system matrices adhere to the pattern of fixed zeros imposed by the structured system. In \cite{lin1974}, Lin characterizes (weak) structural controllability as a graph-theoretic property by associating a graph to the structured system. Following this, there have been a number of papers that deal with structured systems and their properties. The results on (weak) structural controllability have been extended to multi-input systems in \cite{shields1976} and shown to be generic in \cite{glover1976}. Requiring \emph{all} systems with a given structure to be controllable leads to the introduction and characterization of \emph{strong structural controllability} in \cite{mayeda1979, reinschke1992}. More recently, strong structural controllability has been characterized in \cite{jia2020a} for a general class of structured systems whose entries are allowed to be fixed zeros, nonzero or arbitrary real numbers. This is done with the help of a so-called \emph{pattern matrix} with three types of symbolic entries that represents the structure of the structured system.

Some of the most noteworthy applications of the concept of structure are in the context of \emph{networked systems}. The general idea there is that a weighted graph represents the network, where the topology of the graph is known but the precise values of the edge weights are not. Assuming that the edge weights are nonzero real numbers, the networked system can be interpreted as a structured system with a zero/nonzero structure, in the sense that a nonzero entry corresponds to an edge and a zero entry corresponds to the absence of an edge. Then one can study structural properties of the network, i.e., properties that depend solely on its topology. This has already been done for a variety of relevant system theoretic properties, among which are controllability \cite{chapman2013, monshizadeh2014, jia2020b}, \emph{target controllability} \cite{gao2014, li2018, monshizadeh2015, vanwaarde2017, moothedath2019, commault2020} and \emph{input-state observability} \cite{boukhobza2007, kibangou2016, gracy2017, gracy2018}.

In this paper, we will follow the pattern matrix framework as introduced in \cite{jia2020a}. In contrast to \cite{jia2020a}, we will be interested in strong structural controllability of linear descriptor systems, as well as input-state observability and  output controllability of linear systems. We will see that these properties can be characterized elegantly once we have established suitable notions of addition and multiplication for pattern matrices. As such, the main contributions of this paper are as follows. First, we formally define addition and multiplication for pattern matrices and analyse some of their relevant properties. Second, we use these properties to characterize strong structural controllability of linear descriptor systems, and strong structural input-state observability and output controllability of linear systems.

All of these properties have been studied in some form already. Strong structural controllability of linear descriptor systems has been studied in \cite{popli2019} under the name of \emph{selective strong structural controllability}. In comparison, the approach that we take here has the advantages of being conceptually simple and broadly applicable in the study of strong structural properties. In fact, our approach is powerful because it also enables the simple characterization of strong structural output controllability and input-state observability. The latter have been studied in the context of networked systems, where output controllability is typically referred to as target controllability \cite{gao2014, li2018, monshizadeh2015, vanwaarde2017, moothedath2019, commault2020, boukhobza2007, kibangou2016, gracy2017, gracy2018}. However, the results we obtain in this paper are for arbitrary pattern matrices and are thus more general than the results for networked systems, which involve additional structural assumptions \cite{monshizadeh2015, vanwaarde2017, kibangou2016}. Of course, our results can also be used to assess target controllability and input-state observability of networked systems. In this specific setup, we even show (see Example~\ref{ex:target}) that our results on strong structural target controllability can be conclusive while existing results are not.

The outline of this paper is as follows. In Section~\ref{sec:pattern}, we review the concept of pattern matrix and define addition and multiplication for pattern matrices. In Section~\ref{sec:applications}, we apply our results on pattern matrices to the characterization of strong structural controllability of linear descriptor systems, and strong structural input-state observability and output controllability of linear systems. We finish with concluding remarks in Section~\ref{sec:conclusion}.


\section{Pattern matrices}\label{sec:pattern}
In this section, we will review the concept of pattern matrix and then further develop it by defining addition and multiplication for pattern matrices. To begin with, a particular type of pattern matrix was introduced in \cite{jia2020a} in order to formalize the idea of matrices whose entries are not known precisely but are known to be zeros, nonzero or arbitrary real numbers. More precisely, a \emph{pattern matrix} is a matrix with entries from the set of symbols $\{0,\ast, {?}\}$, where $\ast$ represents nonzero real numbers and ${?}$ represents arbitrary real numbers. This is captured in the following definition.
\begin{definition}
	The \emph{pattern class} of the pattern matrix $\calA\in\pattern^{m\times n}$ is defined as the set
	\begin{equation*}
		\calP(\calA) = \left\lbrace A\in\bbR^{m\times n} ~\bigg|~
		\begin{aligned}
			A_{ij} = 0 &\text{ if }\calA_{ij} = {0}\\
			A_{ij} \neq 0 &\text{ if } \calA_{ij} = {*}
		\end{aligned}\right\rbrace.
        \vspace{1mm}
	\end{equation*}
\end{definition}
We can define properties of pattern matrices in terms of the properties of the real matrices in their pattern classes. For example, we say that a pattern matrix $\calA$ has full rank if $A$ has full rank for all $A\in\calP(\calA)$. Rank properties will be crucial in the applications to structured systems since most system-theoretic properties are characterized in terms of full rank conditions. Fortunately, conditions under which a pattern matrix has full row rank exist and can be verified using a simple algorithm (see  \cite[Theorem~11, Lemma~21]{jia2020a}). Naturally, one can check if a matrix has full column rank by checking if its transpose has full row rank.

In practice, we will be working with several ``unknown" matrices that belong to the pattern classes of some known pattern matrices. This will naturally lead to expressions involving sums and products. To understand the results of such expressions, we will define a sensible way of adding and multiplying pattern matrices. Here, sensible means that the result of adding and multiplying pattern matrices gives us some useful information on the result of adding and multiplying matrices belonging to their pattern classes.

To this end, given a pair of pattern matrices,  we want the sum of any pair of real matrices from their pattern classes to be contained in the pattern class of the sum of the pattern matrices. We know that the sum of zero and any real number is just the number itself, while the sum of two nonzero real numbers can be any real number. Motivated by this, we define addition for the set $\pattern$ as shown in Table~\ref{tab:addition}. Then addition for pattern matrices is defined element-wise.
\begin{definition}
	Let $\calA,\calB\in\pattern^{m\times n}$. Their sum $\calA+\calB\in\pattern^{m\times n}$ is defined as\vspace{-.7mm}
	\begin{equation*}
		(\calA + \calB)_{ij} = \calA_{ij} + \calB_{ij}\vspace{-.7mm}
	\end{equation*}
	for all $i\in\{1,\dots,m\}$ and $j\in\{1,\dots n\}$.
\end{definition}
\begin{table}[htbp!]
    \centering
    \begin{tabular}{c|ccc}
        $+$ & $0$ & $*$ & ${?}$\\
        \hline
        $0$ & $0$ & $*$ & ${?}$\tstrut\\
        $*$ & $*$ & ${?}$ & ${?}$\\
        ${?}$ & ${?}$ & ${?}$ & ${?}$
    \end{tabular}
    \caption{Addition for the set $\pattern$.}
    \label{tab:addition}
    \vspace{-11mm}
\end{table}

By definition, if $\calA$ and $\calB$ are pattern matrices of the same dimensions, then $\calP(\calA + \calB) \supset \calP(\calA) + \calP(\calB)$, where
\begin{equation*}
	\calP(\calA) + \calP(\calB) = \set{A + B}{A\in\calP(\calA),\  B\in\calP(\calB)}
\end{equation*}
is the Minkowski sum of sets. It turns out that the converse is true as well.
\begin{proposition}\label{prop:patsum}
	If $\calA$ and $\calB$ are pattern matrices of the same dimensions, then $\calP(\calA + \calB) = \calP(\calA) + \calP(\calB)$.
\end{proposition}
\begin{proof}
	The inclusion $\calP(\calA+\calB) \supset \calP(\calA) + \calP(\calB)$ follows from the definition of addition. For the converse inclusion, let $C\in\calP(\calA+\calB)$ and consider an entry $C_{ij}$. The goal is to show that there exist entries $A_{ij} \in \calP(\calA_{ij})$ and $B_{ij}\in\calP(\calB_{ij})$ such that $C_{ij} = A_{ij} + B_{ij}$. We will consider the cases $C_{ij} = 0$ and $C_{ij} \neq 0$ separately.

	Suppose that $C_{ij} = 0$. Then either $(\calA+\calB)_{ij} = 0$ or $(\calA+\calB)_{ij} = {?}$. In the former, we must have that $\calA_{ij} = 0$ and $\calB_{ij} = 0$, hence $A_{ij} = 0$ and $B_{ij} = 0$ would work. In the latter, there are three cases:
	\begin{enumerate}
		\item If $\calA_{ij}, \calB_{ij} \in\{\ast, {?}\}$, then $-A_{ij}= B_{ij} = 1$.
		\item If $\calA_{ij} = 0$ and $\calB_{ij} = {?}$, then $A_{ij} = B_{ij} = 0$.
		\item If $\calA_{ij} = {?}$ and $\calB_{ij} = {0}$, then $A_{ij} = B_{ij} = 0$.
	\end{enumerate}
	Suppose that $C_{ij} \neq 0$. Then either $(\calA+\calB)_{ij} = *$ or $(\calA+\calB)_{ij} = {?}$. In the former, exactly one of $\calA_{ij}$ and $\calB_{ij}$ is $*$ and the other one is $0$, hence we can pick either $A_{ij} = C_{ij}$ and $B_{ij} = 0$, or $A_{ij}=0$ and $B_{ij} = C_{ij}$. In the latter, there are three cases again:
	\begin{enumerate}
		\item If $\calA_{ij}, \calB_{ij} \in\{\ast, {?}\}$, then $A_{ij}= B_{ij} = \frac{1}{2}C_{ij}$.
		\item If $\calA_{ij} = 0$ and $\calB_{ij} = {?}$, then $A_{ij} = 0$ and $B_{ij} = C_{ij}$.
		\item If $\calA_{ij} = {?}$ and $\calB_{ij} = {0}$, then $A_{ij} = C_{ij}$ and $B_{ij} = 0$.
	\end{enumerate}
	The element $C_{ij}$ was chosen arbitrarily, hence we can always find matrices $A\in\calP(\calA)$ and $B\in\calP(\calB)$ such that $A+B = C$ and thus $\calP(\calA+\calB) \subset \calP(\calA) + \calP(\calB)$.
\end{proof}

In the same vein, we now turn to the definition of multiplication for pattern matrices. Note that the product of zero and any real number is just zero, while the product of two nonzero real numbers is always a nonzero real number. This motivates the definition of multiplication for the set $\pattern$ shown in Table~\ref{tab:multiplication}.
\begin{table}[htbp!]
	\centering
    \vspace{-2mm}
	\begin{tabular}{c|ccc}
		$\cdot$ & $0$ & $*$ & ${?}$\tstrut\\
		\hline
		$0$ & $0$ & $0$ & $0$\tstrut\\
		$*$ & $0$ & $*$ & ${?}$\\
		${?}$ & $0$ & ${?}$ & ${?}$
	\end{tabular}
	\caption{Multiplication for the set $\pattern$.}
	\label{tab:multiplication}
	\vspace{-7mm}
\end{table}
Then we can define pattern matrix multiplication in the usual way.
\begin{definition}
	Let $\calA\in\pattern^{m\times p}$ and $\calB\in\pattern^{p\times n}$. Their product $\calA\calB\in\pattern^{m\times n}$ is defined as\vspace{-1mm}
	\begin{equation*}
		(\calA\calB)_{ij} = \sum_{k=1}^{p} \calA_{ik}\calB_{kj}	\vspace{-1mm}
	\end{equation*}
	for all $i\in\nset{m}$ and $j\in\nset{n}$.
\end{definition}

By definition, if $\calA$ and $\calB$ are of appropriate dimensions, then $\calP(\calA\calB)\supset\calP(\calA)\calP(\calB)$, where \vspace{-1mm}
\begin{equation*}
	\calP(\calA)\calP(\calB) = \set{AB}{A\in\calA,\ B \in \calB}.\vspace{-1mm}
\end{equation*}
Unfortunately, the converse is generally not true. When multiplying matrices with at least two rows or columns, we typically create dependencies between the entries of the product. These dependencies cannot be inferred from the product of pattern matrices.
\begin{example}
	Consider the pattern vectors $\calA = \bbm * & * \ebm \t$ and $\calB = \bbm * & * \ebm$. It is easy to see that
	\begin{equation*}
		\calA\calB = \bbm * & * \\ * & * \ebm \qand \bbm 1 & 1 \\ 1 & 2 \ebm\in\calP(\calA\calB).
	\end{equation*}
	Note that the latter is a matrix of rank 2 and thus it cannot be written as the outer product of two vectors. In other words, the fact that the columns (or rows) of $AB$, where $A\in\calP(\calA)$ and $B\in\calP(\calB)$, are linearly dependent cannot be inferred from the product $\calA\calB$.
\end{example}

Although the equality $\calP(\calA\calB) = \calP(\calA)\calP(\calB)$ does not hold in general, there are special cases of $\mathcal{A}$ and $\mathcal{B}$ for which equality does hold. A notable special case is the one where either $\mathcal{A}$ or $\mathcal{B}$ is the ``identity" pattern matrix $\mathcal{I}$ of appropriate dimensions, defined as a diagonal matrix with $*$'s on the diagonal.
Indeed, suppose that $\mathcal{A} = \mathcal{I}$. It is not difficult to see that $\mathcal{I} \mathcal{B} = \mathcal{B}$ and $\mathcal{P}(\mathcal{I})\mathcal{P}(\mathcal{B}) = \mathcal{P}(\mathcal{B})$, and thus $\calP(\mathcal{I}\calB) = \mathcal{P}(\mathcal{B}) = \calP(\mathcal{I})\calP(\calB)$. In the case where $\mathcal{B} = \mathcal{I}$, we can prove the equality in an analogous way.

Finally, we provide the following lemma, which will prove to be very useful when considering the applications to structured systems in the next section.
\begin{lemma}\label{lem:rank_lambda}
		Let $\calA,\calB\in\pattern^{m\times n}$. Then $A-\lambda B$ has full rank for all $A\in\calP(\calA)$, $B\in\calP(\calB)$ and nonzero $\lambda\in\bbC$ if and only if $\calA + \calB$ has full rank.
\end{lemma}
\begin{proof}
	Suppose that $A - \lambda B$ has full rank for all $A\in\calP(\calA)$, $B\in\calP(\calB)$ and nonzero $\lambda\in\bbC$. Fixing $\lambda = -1$ shows that $A+B$ has full rank for all $A\in\calP(\calA)$, $B\in\calP(\calB)$. But this is equivalent to $C$ having full rank for all $C\in\calP(\calA) + \calP(\calB)$, hence $\calA + \calB$ has full rank due to Proposition~\ref{prop:patsum}.

	Conversely, suppose that $\calA + \calB$ has full rank. We will only treat the case where $m\leq n$ since the case where $n>m$ follows the same reasoning after $\calA$ and $\calB$ are transposed.  With this in mind, let $z\in\bbC^m$ be such that $z\s A - \lambda z\s B = 0$. The goal is to show that $z$ must be the zero vector. Write $z = x + iy$, where $x,y\in\bbR^m$ and $i$ denotes the imaginary unit, and consider $\hat z = x + \alpha y$ with $\alpha\in\bbR$ such that
	\begin{align}
		\alpha&\notin\set{\frac{x_k}{y_k}}{y_k\neq 0,\ k = 1,2,\dots,m}\label{eq:alpha1},\\
		\alpha&\notin\set{\frac{(x\t A)_k}{(y\t A)_k}}{(y\t A)_k\neq 0,\ k = 1,2,\dots,m}\label{eq:alpha2},\\
		\alpha&\notin\set{\frac{(x\t B)_k}{(y\t B)_k}}{(y\t B)_k\neq 0,\ k = 1,2,\dots,m}\label{eq:alpha3}.
	\end{align}
	Note that \eqref{eq:alpha1} implies that $z_k = 0$ if and only if $\hat z_k = 0$. Similarly, \eqref{eq:alpha2} and \eqref{eq:alpha3} imply that $(z\s A)_k = 0$ if and only if $(\hat z\t A)_k = 0$, and $(z\s B)_k = 0$ if and only if $(\hat z\t B)_k = 0$. Furthermore, since $\lambda \neq 0$ and $z\s A = \lambda z\s B$, we have that $(z\s A)_k = 0$ if and only if $(z\s B)_k = 0$, hence $(\hat z\t A)_k = 0$ if and only if $(\hat z\t B)_k = 0$. Therefore, the diagonal matrix $\Delta\in\bbR^{n\times n}$ defined as
	\begin{equation*}
	\Delta_{kk} = \begin{dcases*}
	-1& if $(\hat z\t B)_k = 0$,\\
	-\frac{(\hat z\t A)_k}{(\hat z\t B)_k}& otherwise,
	\end{dcases*}
	\end{equation*}
	is a member of $\mathcal{P}(\mathcal{I})$ and is such that $\hat z\t (A + B\Delta) = 0$. Since $\mathcal{P}(\mathcal{B}) \mathcal{P}(\mathcal{I}) = \mathcal{P}(\mathcal{B})$, it holds that $B\Delta \in\calP(\calB)$ and thus $A + B\Delta \in \calP(\calA+\calB)$ due to Proposition~\ref{prop:patsum}. Then $A+B\Delta$ has full row rank, which implies that $\hat z = 0$ and thus $z = 0$ because of \eqref{eq:alpha1}. This proves the lemma.
\end{proof}


\section{Applications}\label{sec:applications}
In this section, we will show how $\{0,\ast,?\}$ pattern matrices can be used to characterize properties of structured systems. This has already been done for strong structural controllability in \cite{jia2020a}. There it is shown that a structured system is strongly structurally controllable if and only if two system related pattern matrices have full row rank. As the latter can be checked using a simple algorithm, this provides a way to verify strong structural controllability for a given structured system. Here, we will extend the work of \cite{jia2020a} by studying strong structural controllability of linear descriptor systems. In addition, we will also study strong structural input-state observability and output controllability of linear systems. As we will see, addition and multiplication of pattern matrices will play an important role in the study of each of these three properties.

\subsection{Controllability of linear descriptor systems}\label{subsec:dae}
In this subsection, we will extend the results on strong structural controllability from \cite{jia2020a} to linear descriptor systems. Let $\daesys$ denote the system
\begin{equation}\label{eq:dae_system}
	E\dot x(t) = Ax(t) + Bu(t),
\end{equation}
where $t\geq 0$, $x(t)\in\bbR^n$ is the state, $u(t)\in\bbR^m$ is the input, $E\in\bbR^{n\times n}$, $A\in\bbR^{n\times n}$ and $B\in\bbR^{n\times m}$. The system $(E,A,B)$ is called \emph{regular} if $\lambda E-A$ is invertible for some $\lambda\in\bbC$. Regularity of $\daesys$ guarantees the existence and uniqueness of solutions to \eqref{eq:dae_system}, given an initial state and a sufficiently differentiable input (see \cite{dai1989}), hence regularity of $\daesys$ is a desirable property. Typically, the matrix $E$ is singular and thus \eqref{eq:dae_system} puts algebraic constraints on the state. This leads to $\daesys$ having special features that are not found in systems in which the state is not constrained algebraically. Among these are impulse terms, consistent initial conditions, input derivatives in the state trajectory, noncausality, etc. Consequently, there are different kinds of controllability notions defined for $\daesys$, some of which make sense only in the presence of algebraic constraints. We will not go into the analysis of descriptor systems, and will instead focus on a particular definition of controllability and its characterization, as presented in \cite{dai1989}. To this end, let $x(t;x_0, u)$ denote the state trajectory at time $t\geq 0$ for the initial condition $x(0) = x_0\in\bbR^n$ and input $u$.
\begin{definition}{}
	The regular system $\daesys$ is \emph{controllable} if for any $T>0$, $x_0\in\bbR^n$ and $x_1\in\bbR^n$, there exists an input function\footnote{The input is assumed to be in $\daeinputs{h-1}$, the class of ($h-1$)-times piecewise continuously differentiable functions. Here $h$ denotes the index of the descriptor system (see \cite[Chapter~1]{dai1989}).} $u\in \daeinputs{h-1}$ such that $x(T; x_0, u) = x_1$.
\end{definition}

Then we have the following characterization of controllability for regular $\daesys$.
\begin{theorem}\cite[Theorem~2-2.1]{dai1989}
	The regular system $\daesys$ is controllable if and only if
	\begin{equation}\label{eq:rankEB}
		\rank\bbm E & B \ebm = \rank\bbm A - \lambda E & B \ebm = n
	\end{equation}
	for all $\lambda \in\bbC$.
\end{theorem}

Now, suppose that $E$, $A$ and $B$ are not known precisely but are known to belong to the pattern classes of some known pattern matrices. In other words, we know that \daeclassprod{} for given pattern matrices $\calE \in \{0,\ast,?\}^{n \times n}$, $\calA \in \{0,\ast,?\}^{n \times n}$ and $\calB \in \{0,\ast,?\}^{n \times m}$. This naturally leads to a family of systems as $E$, $A$ and $B$ range over the respective pattern classes. This family is completely characterized by $\calE$, $\calA$ and $\calB$, hence we denote it by $\daepatsys$ and refer to it as a \emph{structured system}.

We are interested in conditions under which all regular $(E,A,B) \in \mathcal{P}(\mathcal{E}) \times \mathcal{P}(\mathcal{A}) \times \mathcal{P}(\mathcal{B})$ are controllable. This motivates the following definition.
\begin{definition}
	The structured system $\daepatsys$ is \emph{regularly strongly structurally controllable} if all regular systems $(E,A,B) \in \mathcal{P}(\mathcal{E}) \times \mathcal{P}(\mathcal{A}) \times \mathcal{P}(\mathcal{B})$ are controllable.
\end{definition}

Making use of the results from Section~\ref{sec:pattern}, we now have the following theorem that provides a sufficient condition for regular strong structural controllability.
\begin{theorem}\label{thm:ssdaec}
	The rank conditions \eqref{eq:rankEB} hold for all $(E,A,B) \in \mathcal{P}(\mathcal{E}) \times \mathcal{P}(\mathcal{A}) \times \mathcal{P}(\mathcal{B})$ if and only if 
	\begin{equation*}
		\bbm \calE & \calB \ebm,\quad \bbm \calA &\calB \ebm \quad\text{and}\quad 	\bbm \calA+\calE & \calB \ebm
	\end{equation*}
	have full row rank. Moreover, if these pattern matrices have full row rank then $\daepatsys$ is regularly strongly structurally controllable.
\end{theorem}
\begin{proof}
	Note that the rank conditions \eqref{eq:rankEB} hold for all $(E,A,B) \in \mathcal{P}(\mathcal{E}) \times \mathcal{P}(\mathcal{A}) \times \mathcal{P}(\mathcal{B})$ if and only if
	\begin{equation*}
		\rank\bbm E & B \ebm = \rank\bbm A & B \ebm = \rank \bbm A - \lambda E & B \ebm = n
	\end{equation*}
	for all nonzero $\lambda \in \bbC$, \daeclassprod{}. The result then follows from Lemma~\ref{lem:rank_lambda}.
\end{proof}
\begin{remark}
In the special case where $\mathcal{E} = \mathcal{I}$, all systems $(E,A,B) \in \mathcal{P}(\mathcal{I}) \times \mathcal{P}(\mathcal{A}) \times \mathcal{P}(\mathcal{B})$ are regular. In this case, $E\inv$ always exists and we can also write $\eqref{eq:dae_system}$ as
\begin{equation*}
	\dot x = E\inv A x + E\inv B u.
\end{equation*}
Clearly, $E^{-1} \in \mathcal{P}(\mathcal{I})$ for all $E \in \mathcal{P}(\mathcal{I})$. Since $\mathcal{P}(I) \mathcal{P}(\mathcal{A}) = \mathcal{P}(\mathcal{A})$ and $\mathcal{P}(I) \mathcal{P}(\mathcal{B}) = \mathcal{P}(\mathcal{B})$, we see that regular strong structural controllability of $(\mathcal{I},\mathcal{A},\mathcal{B})$ is equivalent to strong structural controllability of $(\mathcal{A},\mathcal{B})$, as defined in \cite{jia2020a}. In fact, in the special case $\mathcal{E} = \mathcal{I}$, the conditions of Theorem \ref{thm:ssdaec} coincide with the conditions for strong structural controllability given in \cite[Theorem 7]{jia2020a}. To see this, note that $\bbm \mathcal{I} & \calB \ebm$ has full row rank for any $\calB$. In addition, the matrix $\bar\calA := \calA + \mathcal{I}$ is the pattern matrix obtained from $\calA$ by changing the diagonal entries of $\calA$ to
\begin{equation*}
	\bar \calA_{kk} =
	\begin{dcases*}
		\ast & if $\calA_{kk} = 0$,\\
		? & otherwise.
	\end{dcases*}
\end{equation*}
As such, Theorem \ref{thm:ssdaec} requires $\bbm \calA &\calB \ebm$ and $\bbm \bar{\calA} & \calB \ebm$ to have full row rank, which are exactly the two conditions of \cite[Theorem 7]{jia2020a}. These conditions are, in fact, necessary and sufficient for controllability of $(\mathcal{A},\mathcal{B})$. The lack of necessity in the characterization of regular strong structural controllability (Theorem \ref{thm:ssdaec}) stems from the fact that generally not all $(E,A,B) \in \mathcal{P}(\mathcal{E}) \times \mathcal{P}(\mathcal{A}) \times \mathcal{P}(\mathcal{B})$ are regular.
\end{remark}

\subsection{Input-state observability}\label{subsec:iso}
In this section, we will use the techniques developed for the analysis of strong structural controllability to characterize another property, namely, input-state observability. Let $\sys$ denote the system
\begin{align*}
	\dot x(t) &= Ax(t) + Bu(t),\\
	y(t) &= Cx(t) + Du(t),
\end{align*}
where $t\geq 0$ represents time, $x(t)\in\bbR^n$ is the state, $u(t)\in\bbR^m$ is the input, $y(t) \in \mathbb{R}^p$ is the output, $A\in\bbR^{n\times n}$, $B\in\bbR^{n\times m}$, $C\in\bbR^{p\times n}$ and $D\in\bbR^{p\times m}$. For a given initial condition $x(0) = x_0\in\bbR^n$ and input function $u$, we denote the corresponding output trajectory at time $t\geq 0$ by $y(t; x_0, u)$. Then we consider the following definition.
\begin{definition}
	The system $\sys$ is \emph{input-state observable} if $y(t;x_1,u_1) = y(t;x_2,u_2)$ for all $t\geq 0$ implies that $x_1 = x_2$ and $u_1(t) = u_2(t)$ for all $t\geq 0$.
\end{definition}

In other words, a system $\sys$ is input-state observable if different initial conditions and inputs can be distinguished on the basis of the output of the system. Conditions under which this is the case are provided in the following theorem.
\begin{theorem}{\cite[Theorem 3.3]{shali2019}}\label{thm:iso}
	The system $\sys$ is input-state observable if and only if
	\begin{equation*}
		\rank \bbm A - \lambda I & B \\ C & D \ebm = n + m
	\end{equation*}
	for all $\lambda \in\bbC$.
\end{theorem}

As before, instead of considering a single system $\sys$, we consider the family of systems where \classprod{} for given pattern matrices $\calA$, $\calB$, $\calC$ and $\calD$ of appropriate dimensions. We denote this family by $\patsys$ and refer to it as a \emph{structured system}. We are interested in finding necessary and sufficient conditions under which $(A,B,C,D)$ is guaranteed to be input-state observable for all \classprod.
\begin{definition}
	The structured system $\patsys$ is \emph{strongly structurally input-state observable} if $\sys$ is input-state observable for all \classprod{}.
\end{definition}

In view of Theorem~\ref{thm:iso} and the results presented so far, the following characterization of strong structural input-state observability follows naturally.

\begin{theorem}\label{thm:ssiso}
	The structured system $\patsys$ is strongly structurally input-state observable if and only if 
	\begin{equation}\label{eq:ssiso_pat_marices}
		\bbm \calA & \calB\\ \calC & \calD \ebm \qand \bbm \calA + \calI & \calB \\ \calC & \calD \ebm
	\end{equation}
	have full column rank.
\end{theorem}
\begin{proof}
	We claim that $\patsys$ is strongly structurally input-state observable if and only if
	\begin{equation}\label{eq:rank_delta}
		\rank \bbm A - \lambda \Delta& B\\ C & D \ebm = n + m
	\end{equation}
	for all $\lambda\in\bbC$, $\Delta\in\calP(\calI)$, \classprod. Indeed, \eqref{eq:rank_delta} holds if and only if\vspace{-.7mm}
	\begin{equation*}
		\rank \bbm \Delta\inv A - \lambda I & \Delta\inv B \\ C & D \ebm = n+m,
		\vspace{-.7mm}
	\end{equation*}
	where we have $\Delta\inv A\in\calP(\calA)$ and $\Delta\inv B\in\calP(\calB)$ since $\Delta\inv \in\calP(\calI)$, $\calP(\calI)\calP(\calA) = \calP(\calA)$ and $\calP(\calI)\calP(\calB) = \calP(\calB)$. Therefore, $\patsys$ is strongly structurally input-state observable if and only if \vspace{-.7mm}
	\begin{equation*}
		\rank\bbm A & B \\ C & D \ebm = \rank\bbm A & B \\ C & D \ebm - \lambda \bbm \Delta & 0 \\ 0 & 0 \ebm  = n+m
		\vspace{-.7mm}
	\end{equation*}
	for all nonzero $\lambda \in \bbC$, $\Delta\in\calP(\calI)$, \classprod{}. In view of Lemma~\ref{lem:rank_lambda}, the latter holds if and only if the pattern matrices in \eqref{eq:ssiso_pat_marices} have full column rank.
\end{proof}

\subsection{Output controllability}\label{subsec:oc}
In this section, we will show how pattern matrix multiplication and its properties can be used to characterize strong structural output controllability. To this end, consider the system $\sys$ as defined in Section~\ref{subsec:iso}.
\begin{definition}
	The system $\sys$ is \emph{output controllable} if for any $x_0\in\bbR^n$ and $y_1 \in\bbR^p$, there exist a time $T>0$ and an input $u$ such that $y(T;x_0,u) = y_1$.
\end{definition}

The following is a well-known characterization of output controllability of $\sys$, c.f., \cite[Exercise 3.22]{trentelman2001}.
\begin{theorem}\label{thm:oc}
	The system $\sys$ is output controllable if and only if
	\begin{equation*}
		\rank \bbm D & CB & CAB & \cdots & CA^{n-1}B \ebm = p.
	\end{equation*}
\end{theorem}
\vspace{1mm}

Now, consider the structured system $\patsys$ as defined in Section~\ref{subsec:iso}.
\begin{definition}
	The structured system $\patsys$ is \emph{strongly structurally output controllable} if $\sys$ is output controllable for all \classprod.
\end{definition}

We are interested in conditions under which $\patsys$ is strongly structurally output controllable. Note that the condition for output controllability of $\sys$ involves products of system matrices, unlike the conditions for controllability of $\daesys$ or input-state observability of $\sys$. This suggest that we need to consider products of pattern matrices when investigating strong structural output controllability of $\patsys$. Unfortunately, since products of pattern matrices do not share the same favourable property as sums, i.e., $\calP(\calA\calB) \neq \calP(\calA)\calP(\calB)$, we cannot easily derive necessary and sufficient conditions. Nevertheless, we state and prove the following sufficient condition.
\begin{theorem}\label{thm:ssoc}
	The structured system $\patsys$ is strongly structurally output controllable if\vspace{-1mm}
	\begin{equation*}
		\bbm \calD & \calC\calB & \calC\calA\calB & \cdots & \calC\calA^{n-1}\calB \ebm\vspace{-1mm}
	\end{equation*}
	has full row rank.
\end{theorem}
\begin{proof}
	Let \classprod. Recall that $\calP(\calC)\calP(\calB) \subset \calP(\calC\calB)$, that is, $CB\in\calP(\calC\calB)$ for all $C\in\calP(\calC)$ and $B\in\calP(\calB)$. By induction, it follows that $\calP(\calC)\calP(\calA)^k\calP(\calB) \subset \calP(\calC\calA^k\calB)$ for all positive integers $k$. In other words, we have that\vspace{-2mm}
	\begin{multline*}
   		\bbm D & CB & CAB & \cdots & CA^{n-1}B \ebm\\ \subset \calP(\bbm \calD & \calC\calB & \calC\calA\calB & \cdots & \calC\calA^{n-1}\calB \ebm),\vspace{-1mm}
	\end{multline*}
	hence $\sys$ is output controllable due to Theorem~\ref{thm:oc}. As $A$, $B$, $C$ and $D$ were chosen arbitrarily, it follows that $\patsys$ is strongly structurally output controllable.
\end{proof}

As already mentioned in the introduction, strong structural output controllability is closely related to strong structural target controllability of networked systems. To show this, we will follow the exposition in \cite{monshizadeh2015}. Consider the graph $G = (V,E)$ with vertex set $V=\{1,\dots,n\}$ and edge set $E\subset V\times V$. The \emph{qualitative class} $Q(G)$ of $G$ is defined as
\begin{equation*}
	Q(G) = \set{A\in\bbR^{n\times n}}{\text{for }i\neq j,\ A_{ij}\neq 0 \Leftrightarrow (j,i)\in E}.
\end{equation*}
For subsets $V_r, V_c\subset V$, let $P(V_r;V_c)$ denote the submatrix of the $n\times n$ identity matrix whose rows and columns are indexed by $V_r$ and $V_c$, respectively. Now, consider a \emph{leader set} $V_L\subset V$ and a \emph{target set} $V_T\subset V$. The triple $(G; V_L; V_T)$ defines the family of systems $(A,\bar{B},\bar{C},0)$, where $A\in Q(G)$, $\bar B = P(V;V_L)$ and $\bar C = P(V_T;V)$. The triple $(G; V_L; V_T)$ is said to be \emph{strongly structurally target controllable} if $(A,\bar{B},\bar{C},0)$ is output controllable for all $A\in Q(G)$. This already suggests a connection between strong structural target controllability and strong structural output controllability. To make this explicit, let $\calA\in\pattern^{n\times n}$ be such that
\begin{equation*}
	\calA_{ij} =
	\begin{dcases*}
		? & if $i = j$,\\
		* & if $i\neq j$ and $(j,i)\in E$,\\
		0 & otherwise,
	\end{dcases*}
\end{equation*}
and note that $\calP(\calA) = Q(G)$. Moreover, let $\calB$ and $\calC$ be the pattern matrices obtained from $\bar B$ and $\bar C$ by replacing all $1$'s with $*$'s. Given the special structure of $\bar B$ and $\bar C$, any matrix $B\in\calP(\calB)$ can be obtained from $\bar B$ by an appropriate nonzero scaling of its columns and any matrix $C\in\calP(\calC)$ can be obtained from $\bar C$ by an appropriate nonzero scaling of its rows. Since the rank of a matrix is invariant under nonzero scaling of its rows and columns, it follows that
\begin{equation*}
	\rank \bbm \bar{C}\bar{B} & \bar{C}A\bar{B} & \cdots & \bar{C}A^{n-1}\bar{B} \ebm = p
\end{equation*}
for all $A\in Q(G)$ if and only if
\begin{equation*}
	\rank \bbm C B & C A C & \cdots & C A^{n-1} B \ebm  = p
\end{equation*}
for all $A\in\calP(\calA)$, $B \in \calP(\calB)$ and $C \in \calP(\calC)$. This implies that $(G;V_L;V_T)$ is strongly structurally target controllable if and only if $(\calA,\calB,\calC, 0)$ is strongly structurally output controllable, hence we can use Theorem~\ref{thm:ssoc} to check for strong structural target controllability. In fact, Theorem~\ref{thm:ssoc} can reveal that $(G;V_L;V_T)$ is strongly structurally target controllable in cases where the theorems in \cite{monshizadeh2015} are inconclusive. We demonstrate this in the following example, which is borrowed from \cite{monshizadeh2015}.
\begin{example}\label{ex:target}
	Consider the graph $G(V,E)$ depicted in Figure~\ref{fig:network}. Let $V_L = \{1,2\}$ and $V_T=\{1,\dots,7\}$ be the leader and targets sets, respectively.  Let $\calA$ be the pattern matrix for which $\calP(\calA) = Q(G)$ and let $\calB$ and $\calC$ be the pattern matrices obtained from $P(V;V_L)$ and $P(V_T;V)$ by replacing all $1$'s with $*$'s. With $\calD = 0$, the matrix in Theorem~\ref{thm:ssoc} is given by\vspace{.7mm}
	\begin{equation*}
		\bbm
		0 & 0 & * & 0 & ? & * & ? & ? & ? & \cdots\ \\
		0 & 0 & 0 & * & * & ? & ? & ? & ? & \cdots\ \\
		0 & 0 & 0 & 0 & * & * & ? & ? & ? & \cdots\ \\
		0 & 0 & 0 & 0 & 0 & * & ? & ? & ? & \cdots\ \\
		0 & 0 & 0 & 0 & 0 & 0 & * & ? & ? & \cdots\ \\
		0 & 0 & 0 & 0 & 0 & 0 & 0 & * & ? & \cdots\ \\
		0 & 0 & 0 & 0 & 0 & 0 & 0 & 0 & * & \cdots
		\ebm,\vspace{.7mm}
	\end{equation*}
	which has full row rank due to the upper triangular structure when the first two columns are neglected. In view of Theorem~\ref{thm:ssoc} and the discussion above, we conclude that $(G;V_L;V_T)$ is strongly structurally target controllable. Note that the authors of \cite{monshizadeh2015} could not make this conclusion, as explained in the last paragraph of \cite[Section~VI]{monshizadeh2015}.
\end{example}

\begin{figure}
	\centering
	\scalebox{0.75}{
	\begin{tikzpicture}[scale = 0.8, ->,>=stealth',shorten >=1pt,auto,node distance=2.8cm,
		semithick]
		\tikzset{VertexStyle1/.style = {shape = circle,
				color=black,
				fill=white!96!black,
				minimum size=0.5cm,
				text = black,
				inner sep = 2pt,
				outer sep = 1pt,
				minimum size = 18pt,
				draw}
		}
		\tikzset{VertexStyle2/.style = {shape = circle,
				color=black,
				fill=black!60!white,
				minimum size=0.5cm,
				text = white,
				inner sep = 2pt,
				outer sep = 1pt,
				minimum size = 18pt,
				draw}
		}

		\foreach \name/\x/\y in {3/2/-1, 4/2/1, 5/4/-1, 6/4/1, 7/5/0, 8/6/1, 9/6/-1}
		\node[VertexStyle1] (\name) at (\x,\y) {$\name$};

		\foreach \name/\x/\y in {1/0/-1, 2/0/1}
		\node[VertexStyle2] (\name) at (\x,\y) {$\name$};

		\foreach \from/\to in {1/3, 2/4, 3/5, 4/6, 5/7, 2/3, 4/5}
		\draw (\from) -- (\to);

		\foreach \from/\to in {}
		\draw[dashed] (\from) -- (\to);

		\draw (2) to [out=-65, in=65] (1);
		\draw (1) to [out=115, in=-115] (2);
		\draw (4) to [out=-65, in=65] (3);
		\draw (3) to [out=115, in=-115] (4);

        \draw (3) to [out=-30,in=-150,looseness=1] (9);
        \draw (4) to [out=30,in=150,looseness=1] (8);
	\end{tikzpicture}}
    \vspace{-2mm}
	\caption{The graph $G=(V,E)$.}
	\label{fig:network}
    \vspace{-7mm}
\end{figure}
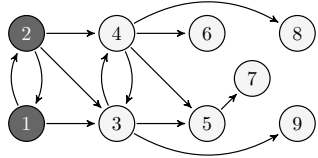


\section{Conclusion}\label{sec:conclusion}
In this paper, we adopted and expanded the framework for strong structural controllability introduced in \cite{jia2020a} in order to study structured systems and their strong structural properties. In particular, we defined addition and multiplication for the set of symbols $\pattern$, which allowed us to define addition and multiplication for pattern matrices with entries in the set $\pattern$. These definitions are such that the pattern class of a sum (product) of pattern matrices is contained in the sum (product) of their pattern classes. We showed that the converse is true for sums, but generally not true for products. Using these operations and their properties, we extended the results on strong structural controllability from \cite{jia2020a} to linear descriptor systems. Furthermore, we characterized strong structural input-state observability and output controllability of linear systems. We also showed that our results on strong structural output controllability can be used in the context of networked systems to verify strong structural target controllability. We note that by defining the graph associated to a pattern matrix (see \cite{jia2020a}), the algebraic characterizations presented in this paper can be translated to graph-theoretic ones, which can be more insightful in the context of networked systems. Finally, it would be worthwhile to investigate extensions of our work to more general classes of structured systems, e.g., structured systems that allow given nonzero or arbitrary entries to be constrained to take identical values (see \cite{jia2020b, jia2020c}).

\bibliographystyle{ieeetr}
\bibliography{references}

\begin{thebibliography}{10}

\bibitem{lin1974}
C.-T. Lin, ``Structural controllability,'' {\em IEEE Transactions on Automatic
  Control}, vol.~19, no.~3, pp.~201--208, 1974.

\bibitem{shields1976}
R.~Shields and J.~Pearson, ``Structural controllability of multiinput linear
  systems,'' {\em IEEE Transactions on Automatic Control}, vol.~21, no.~2,
  pp.~203--212, 1976.

\bibitem{glover1976}
K.~Glover and L.~Silverman, ``Characterization of structural controllability,''
  {\em IEEE Transactions on Automatic Control}, vol.~21, no.~4, pp.~534--537,
  1976.

\bibitem{mayeda1979}
H.~Mayeda and T.~Yamada, ``Strong structural controllability,'' {\em SIAM
  Journal on Control and Optimization}, vol.~17, no.~1, pp.~123--138, 1979.

\bibitem{reinschke1992}
K.~Reinschke, F.~Svaricek, and H.-D. Wend, ``On strong structural
  controllability of linear systems,'' in {\em Proceedings of the IEEE
  Conference on Decision and Control}, vol.~1, pp.~203--208, 1992.

\bibitem{jia2020a}
J.~{Jia}, H.~J. {Van Waarde}, H.~L. {Trentelman}, and M.~K. {Camlibel}, ``A
  unifying framework for strong structural controllability,'' {\em IEEE
  Transactions on Automatic Control}, pp.~1--1, 2020.

\bibitem{chapman2013}
A.~{Chapman} and M.~{Mesbahi}, ``On strong structural controllability of
  networked systems: A constrained matching approach,'' in {\em American
  Control Conference}, pp.~6126--6131, 2013.

\bibitem{monshizadeh2014}
N.~{Monshizadeh}, S.~{Zhang}, and M.~K. {Camlibel}, ``Zero forcing sets and
  controllability of dynamical systems defined on graphs,'' {\em IEEE
  Transactions on Automatic Control}, vol.~59, no.~9, pp.~2562--2567, 2014.

\bibitem{jia2020b}
J.~{Jia}, H.~L. {Trentelman}, W.~{Baar}, and M.~K. {Camlibel}, ``Strong
  structural controllability of systems on colored graphs,'' {\em IEEE
  Transactions on Automatic Control}, vol.~65, no.~10, pp.~3977--3990, 2020.

\bibitem{gao2014}
J.~Gao, Y.~Y. Liu, R.~M. D'Souza, and A.~L. Barab\'{a}si, ``Target control of
  complex networks,'' {\em Nature Communications}, vol.~5, 2014.

\bibitem{li2018}
J.~Li, X.~Chen, S.~Pequito, G.~J. Pappas, and V.~M. Preciado, ``Structural
  target controllability of undirected networks,'' in {\em Proceedings of the
  IEEE Conference on Decision and Control}, pp.~6656--6661, 2018.

\bibitem{monshizadeh2015}
N.~Monshizadeh, M.~K. Camlibel, and H.~L. Trentelman, ``Strong targeted
  controllability of dynamical networks,'' in {\em Proceedings of the IEEE
  Conference on Decision and Control}, pp.~4782--4787, 2015.

\bibitem{vanwaarde2017}
H.~J. van Waarde, M.~K. Camlibel, and H.~L. Trentelman, ``A distance-based
  approach to strong target control of dynamical networks,'' {\em IEEE
  Transactions on Automatic Control}, vol.~62, no.~12, pp.~6266--6277, 2017.

\bibitem{moothedath2019}
S.~{Moothedath}, K.~{Yashashwi}, P.~{Chaporkar}, and M.~N. {Belur}, ``Target
  controllability of structured systems,'' in {\em Proceedings of the European
  Control Conference}, pp.~3484--3489, 2019.

\bibitem{commault2020}
C.~{Commault}, J.~{van der Woude}, and P.~{Frasca}, ``Functional target
  controllability of networks: Structural properties and efficient
  algorithms,'' {\em IEEE Transactions on Network Science and Engineering},
  vol.~7, no.~3, pp.~1521--1530, 2020.

\bibitem{boukhobza2007}
T.~Boukhobza, F.~Hamelin, and S.~Martinez-Martinez, ``State and input
  observability for structured linear systems: A graph-theoretic approach.,''
  {\em Automatica}, vol.~43, no.~7, pp.~1204--1210, 2007.

\bibitem{kibangou2016}
A.~Y. Kibangou, F.~Garin, and S.~Gracy, ``Input and state observability of
  network systems with a single unknown input,'' {\em IFAC-PapersOnLine},
  vol.~49, no.~22, pp.~37--42, 2016.

\bibitem{gracy2017}
S.~Gracy, F.~Garin, and A.~Y. Kibangou, ``Strong structural input and state
  observability of {LTV} network systems with multiple unknown inputs,'' {\em
  IFAC-PapersOnLine}, vol.~50, no.~1, pp.~7357--7362, 2017.

\bibitem{gracy2018}
S.~Gracy, F.~Garin, and A.~Kibangou, ``Structural and strongly structural input
  and state observability of linear network systems,'' {\em IEEE Transactions
  on Control of Network Systems}, vol.~5, no.~4, pp.~2062--2072, 2018.

\bibitem{popli2019}
N.~Popli, S.~Pequito, S.~Kar, A.~P. Aguiar, and M.~Ili\'c, ``Selective strong
  structural minimum-cost resilient co-design for regular descriptor linear
  systems,'' {\em Automatica}, vol.~102, pp.~80--85, 2019.

\bibitem{dai1989}
L.~Dai, {\em Singular Control Systems}, vol.~118 of {\em Lecture Notes in
  Control and Information Sciences}.
\newblock Springer-Verlag Berlin Heidelberg, 1~ed., 1989.

\bibitem{shali2019}
B.~Shali, ``Strong structural properties of structured linear systems,'' {\em
  Master's thesis, University of Groningen}, 2019.

\bibitem{trentelman2001}
H.~L. Trentelman, A.~A. Stoorvogel, and M.~L.~J. Hautus, {\em Control Theory
  for Linear Systems}.
\newblock Springer-Verlag London, 1~ed., 2001.

\bibitem{jia2020c}
J.~{Jia}, H.~L. {Trentelman}, N.~{Charalampidis}, and M.~{Kanat Camlibel},
  ``{Strong Structural Controllability of Colored Structured Systems},'' {\em
  arXiv e-prints}, p.~arXiv:2003.02168, 2020.

\end{thebibliography}
\end{document}